\documentclass[12pt]{amsart}
 \usepackage{amscd}
 \usepackage{amssymb}
 \usepackage{amsmath}
 \usepackage{mathrsfs}
 \usepackage{xcolor}
 \usepackage{hyperref}
 \usepackage[all]{xypic}
 \usepackage{enumitem,lipsum}
 \usepackage{mathtools}
  \usepackage[a4paper,
            bindingoffset=0.2in,
            left=.8in,
            right=.8in,
            top=1in,
            bottom=.8in,
            footskip=.5in]{geometry}

\usepackage{blindtext}

 \def\inf{\operatorname{inf}}
 \def\sup{\operatorname{sup}}
 \def\max{\operatorname{max}}

 \newtheorem{lemma}{Lemma}[section]
 
 \newtheorem{theorem}[lemma]{Theorem}
 \newtheorem{proposition}[lemma]{Proposition}
 
 \newtheorem{definition}[lemma]{Definition}

\begin{document}

%\begin{frontmatter}

\baselineskip=18pt
%\begin{frontmatter}
%\vspace*{-1cm}
\title[Characterizations of strictly convex spaces]{Characterizations of strictly convex spaces and proximal uniform normal structure}

%\title{On proximal uniform normal structure and relatively nonexpansive mappings}
%\tnoteref{mytitlenote}}
%\tnotetext[mytitlenote]{Fully documented templates are available in the elsarticle package on \href{http://www.ctan.org/tex-archive/macros/latex/contrib/elsarticle}{CTAN}.}

%% Group authors per affiliation:
\author{Abhik Digar}
%\address[add1]{Department of Mathematics, Indian Institute of Technology Ropar, Punjab-140 001, India.}
\address{Indian Institute of Technology Kanpur, Department of Mathematics \& Statistics, Kalyanpur, Uttar Pradesh, 208016, India}
\email{abhikdigar@gmail.com}

%\author[add1]{Abhik Digar}
%\address[add1]{Department of Mathematics, Indian Institute of Technology Ropar, Punjab-140 001, India.}
%\ead{abhikdigar@gmail.com}
\author{G. Sankara Raju Kosuru}
%\address[add1]{Department of Mathematics, Indian Institute of Technology Ropar, Punjab-140 001, India.}
\address{Indian Institute of Technology Punjab, Rupnagar, Punjab 140 001, India}
\email{raju@iitrpr.ac.in}

%
%\author[add1]{G. Sankara Raju Kosuru}
%%\address[add1]{Department of Mathematics, Indian Institute of Technology Ropar, Punjab-140 001, India.}
%\ead{raju@iitrpr.ac.in}

%\author{G. Sankara Raju Kosuru\fnref{myfootnote}}
%\address{Department of Mathematics, Indian Institute of Technology Ropar, Punjab-140 001, India.}
%\fntext[myfootnote]{Since 1880.}

%% or include affiliations in footnotes:
%\author[mymainaddress,mysecondaryaddress]{Elsevier Inc}
%\ead[url]{www.elsevier.com}

%\author[Abhik Digar]{\corref{mycorrespondingauthor}}
%\cortext[Abhik Digar]{Corresponding author}
%\ead{abhikdigar@gmail.com}

%\address[mymainaddress]{1600 John F Kennedy Boulevard, Philadelphia}
%\address[mysecondaryaddress]{360 Park Avenue South, New York}


\begin{abstract}
We provide a few characterizations of a strictly convex Banach space. Using this we improve the main theorem of [Digar, Abhik; Kosuru, G. Sankara Raju; Cyclic uniform Lipschitzian mappings and proximal uniform normal structure. Ann. Funct. Anal. 13 (2022)].
%, no. 1, Paper No. 5, 14 pp.]
\end{abstract}
%
%\begin{keyword}
%%best proximity pairs \sep relatively nonexpansive mapping \sep proximal uniform normal structure \sep Kadec Klee property \sep approximatively compact
%%\MSC[2010] 46B20\sep 37C25\sep 47H09
%\end{keyword}
%
%\end{frontmatter}
%
%\linenumbers
%
%\section{The Elsevier article class}
\maketitle
\section{Introduction and preliminaries}\label{Section1}
We first fix a few notations, which will be used in the sequel. In the next section, we prove charectrizations of strictly convex spaces and using one of these characterizations, we fix a gap in the main theorem (Theorem 5.8) of \cite{AbhikRaju2021}. Let $\mathcal{E}$ be a Banach space and let $(\mathcal{P},\mathcal{Q})$ be a pair of non-empty subsets of $\mathcal{E}.$ We denote the sets $\mathcal{P}_0=\{p\in \mathcal{P}: \|p-q\|=d(\mathcal{P}, \mathcal{Q})~\mbox{for some}~q\in \mathcal{Q}\}$ and $\mathcal{Q}_0=\{q\in \mathcal{Q}: \|p-q\|=d(\mathcal{P}, \mathcal{Q})~\mbox{for some}~p\in \mathcal{P}\},$ here $d(\mathcal{P}, \mathcal{Q})=\inf\{\|p-q\|: p\in \mathcal{P}, a\in \mathcal{Q}\}.$ Then $(\mathcal{P}, \mathcal{Q})$ is said to be proximal if $\mathcal{P}=\mathcal{P}_0$ and $\mathcal{Q}=\mathcal{Q}_0.$ Also, $(\mathcal{P}, \mathcal{Q})$ is said to be semisharp proximal (\cite{Raju2011}) if for any $w$ in $\mathcal{P}$ (similarly, in $\mathcal{Q}$), there exists $u, v$ in $\mathcal{Q}$ (similarly, in $\mathcal{P}$) such that $\|w-u\|=d(\mathcal{P}, \mathcal{Q})=\|w-v\|,$ then it is the case that $u=v.$ In this case, we denote $u$ by $w'$ (as well as $w$ by $u'$). A proximal pair is known as sharp proximal if it is semisharp. Further, the sharp proximal pair $(\mathcal{P}, \mathcal{Q})$ is parallel (\cite{Espinola2008}) if $\mathcal{Q}=\mathcal{P}+h$ for some $h\in \mathcal{E}.$ In this case, we have $p'=p+h,~q'=q-h$ for $p\in \mathcal{P}, q\in \mathcal{Q}$ and $\|h\|=d(\mathcal{P},\mathcal{Q}).$
 Suppose $\mathcal{D}$ is a non-empty subset of $\mathcal{E}.$ The metric projection $P_\mathcal{D}: \mathcal{E}\to \mathcal{D}$ is defined by $P_\mathcal{D}(x)=\{z\in \mathcal{D}: \|x-z\|=d(x, \mathcal{D})\}$ for all $x\in \mathcal{E}.$ It is known that if $\mathcal{D}$ is weakly compact and convex subset of a strictly convex Banach space, then $P_\mathcal{D}(x)$ is singleton for every $x\in \mathcal{E}.$ For $z\in \mathcal{D}$ and $r>0,$ the closed ball centered at $x$ and radius $r$ is denoted by $ \mathbb{\mathbb{B}}(x,r).$
We recall a lemma from \cite{AbhikRaju2021}.
\begin{lemma}\label{Lemma:AbhikRaju2021}
Let $(\mathcal{P},\mathcal{Q})$ be a non-empty closed convex pair in a Banach space with $(\mathcal{P},\mathcal{Q})$ semisharp proximal. Suppose that  $\{u_n\}$ is a sequence in $\mathcal{P}$ and $(u,v)$ is  a proximal pair in $(\mathcal{P},\mathcal{Q})$ such that $\displaystyle\lim_{n\to \infty}\|u_n-v\|=d(\mathcal{P}, \mathcal{Q}).$ Then $\displaystyle\lim_{n\to \infty} u_n=u.$
\end{lemma}
The following two results were proved in \cite{EldredRaj2014}.
 
 \begin{theorem}\label{Thm1:EldredRaj2014}
 Let $\mathcal{C}$ be a non-empty closed convex subset of a strictly convex Banach space. Then $\mathcal{C}$ contains not more than one point of minimum norm.
 \end{theorem}

\begin{theorem}\label{Thm2:EldredRaj2014}
Let $(\mathcal{P},\mathcal{Q})$ be a non-empty closed convex pair in a strictly convex Banach space. Then the restriction of the metric projection operator $P_{\mathcal{P}_0}$ to $\mathcal{Q}_0$ is an isometry.
\end{theorem}

\section{Characterizations of strictly convex spaces}\label{Section2}
 Let $(\mathcal{P}, \mathcal{Q})$ be a non-empty pair in  a Banach space $\mathcal{E}.$ For $w\in \mathcal{E},$ a sequence $\{p_n\}$ in $\mathcal{P}$ is said to be minimizing with respect to $w$ if $\displaystyle\lim_{n\to \infty}\|p_n-w\|=d(w,\mathcal{P}).$ Also, the sequence is said to be minimizing with respect to $\mathcal{Q}$ if $d(p_n, \mathcal{Q})\to d(\mathcal{P}, \mathcal{Q}).$ 
 It is well-known that $\mathcal{P}$ is approximatively compact if every minimizing sequence in $\mathcal{P}$ has a convergent subsequence (\cite{book:Megginson}). 
 We say that $\mathcal{P}$ is weak approximatively compact with respect to $\mathcal{Q}$ if every minimizing sequence in $\mathcal{P}$ with respect to $\mathcal{Q}$ has a weak convergent subsequence in $\mathcal{P}.$ The pair $(\mathcal{P}, \mathcal{Q})$ is weak approximatively compact if $\mathcal{P}$ and $\mathcal{Q}$ are weak approximatively compact with respect to $\mathcal{Q}$ and $\mathcal{P}$ respectively.
The pair $(\mathcal{P}, \mathcal{Q})$ is said to have the property UC (\cite{Suzuki2009}) if $\{u_n\}, \{v_n\}$ are sequences in $\mathcal{P}$ and $\{q_n\}$ is a sequence in $\mathcal{Q}$ with $\displaystyle \lim_{n\to \infty} \|u_n-q_n\|=d(\mathcal{P}, \mathcal{Q})=\displaystyle \lim_{n\to \infty} \|v_n-q_n\|,$ then it is the case that $\displaystyle\lim_{n\to \infty} \|u_n- v_n\|=0.$
 It is easy to see that if a non-empty pair has property UC, then it is semisharp proximal. 
The Banach space $\mathcal{E}$ is said to have the Kadec-Klee property (\cite{book:Megginson}) if for any sequence on the unit sphere $S_\mathcal{E}$ the weak and strong convergence concide. 

A normed linear space is said to be strictly convex if its unit sphere contains no non-trivial line segement. The following result provides the characterizations of a strictly convex Banach space.
\begin{theorem}\label{Thm2.1:StrictConvexityEquivalence}
Let $\mathcal{E}$ be a Banach space. Then the following statements are equivalent:
\begin{enumerate}
\item[(a)] $\mathcal{E}$ is strictly convex.
\item[(b)] Suppose $(\mathcal{P},\mathcal{Q})$ is a non-empty closed bounded convex pair in  $\mathcal{E}$. Then $(\mathcal{P},\mathcal{Q})$ is semisharp proximal.
\item[(c)] If $(\mathcal{P},\mathcal{Q})$ is a non-empty closed bounded convex pair in  $\mathcal{E}$ and $(p,q)\in (\mathcal{P},\mathcal{Q}).$ Suppose $\{p_n\}$ is a sequence in $\mathcal{P}$ such that $\displaystyle\lim_{n\to \infty}\|p_n-q\|=d(\mathcal{P}, \mathcal{Q})=\|p-q\|.$ Then $\displaystyle\lim_{n\to \infty}p_n=p.$
\item[(d)] If $\mathcal{P}$ and $\mathcal{Q}$ are compact subsets of $\mathcal{E}$ with $\mathcal{P}$ convex, then $(\mathcal{P},\mathcal{Q})$ has property UC.
\item[(e)] Suppose $(\mathcal{P},\mathcal{Q})$ is a non-empty closed bounded convex pair in  $\mathcal{E}$. If $\mathcal{E}$ has the Kadec-Klee property and $(\mathcal{P},\mathcal{Q})$ is weak approximatively compact, then $(\mathcal{P},\mathcal{Q})$ has property UC.
\item[(f)] Suppose $(\mathcal{P},\mathcal{Q})$ is a non-empty closed bounded convex pair in  $\mathcal{E}$. For $p,w\in \mathcal{P}$ and $q\in \mathcal{Q},$ with $\|p-q\|=d(q,\mathcal{P})$ and $\|w-q\|=d(q,\mathcal{P}),$ we have $p=w.$
\item[(g)] If $\mathcal{P}$ is a non-empty closed convex subset of  $\mathcal{E},$ then it has no more than one element of minimum norm.
\item[(h)] If $(\mathcal{P},\mathcal{Q})$ is a non-empty closed bounded convex pair in  $\mathcal{E}$ and $\mathcal{P}_0\neq \emptyset$, then the restriction of the metric projection $P_{\mathcal{P}_0}: \mathcal{Q}_0\to \mathcal{P}_0$ is an isometry. 
\item[(i)] Let $(\mathcal{P},\mathcal{Q})$ be a non-empty closed bounded convex pair in  $\mathcal{E}$. For any $p\in \mathcal{P},$ the closed ball $\mathbb{\mathbb{B}}(p,d(\mathcal{P},\mathcal{Q}))$ intersects $\mathcal{Q}$ at not more than one point. 
%Moreover, if it intersects at $q\in \mathcal{Q},$ then $q=p'.$
\end{enumerate}
\end{theorem}
\begin{proof} 
 We only prove the following implications: $(b)\Rightarrow (a), (b)\Leftrightarrow (d), (b)\Leftrightarrow (e), (a)\Leftrightarrow (f), (a)\Leftrightarrow (g), (a)\Leftrightarrow (h), (a)\Leftrightarrow (i).$ The proof of $(a) \Rightarrow (b),~(c) \Rightarrow (b)$ are obvious and $(b) \Rightarrow (c)$ follows from Lemma \ref{Lemma:AbhikRaju2021}. 
  
\noindent $(b)\Rightarrow (a):$ Suppose $\mathcal{E}$ is not strictly convex. Then there exists a nontrivial line segment $\mathcal{Q}=\{tc_1+(1-t)c_2: 0\leq t\leq 1\}$ in $S_\mathcal{E},$ where $c_1, c_2\in S_\mathcal{E}$ with $c_1\neq c_2.$ Set $\mathcal{P}=\{0\}.$ Then $d(\mathcal{P}, \mathcal{Q})=1.$ Clearly $\mathcal{Q}$ is non-empty closed convex subset of $\mathcal{E}$ such that $\left\|0-\left(tc_1+(1-t)c_2\right)\right\|=1$ for all $0\leq t\leq 1.$ This contradicts the semisharp proximality of $(\mathcal{P}, \mathcal{Q}).$ Thus we have $(b) \Rightarrow (a).$

\noindent $(b) \Leftrightarrow  (d):$ Let $(\mathcal{P}, \mathcal{Q})$ be a compact pair in a strictly convex Banach space with $\mathcal{P}$ convex. Suppose there exist sequences $(u_n), (z_n)$ in $\mathcal{P}$ and $(v_n)$ in $\mathcal{Q}$ such that $\max\{\|u_n-v_n\|, \|z_n-v_n\|\}\to d(\mathcal{P},\mathcal{Q}).$ Choose a subsequence $(n_j)$ and points $u, z\in \mathcal{P}, v\in \mathcal{Q}$ such that $u_{n_j}\to u, z_{n_j}\to z$ and $v_{n_j}\to v.$ Then $\|u-v\|=d(\mathcal{P},\mathcal{Q})=\|z-v\|.$ By $(b),~u=z.$ We now prove that $\|u_n-z_n\|\to 0.$ If not, then there exist $\epsilon_0>0$ and a subsequence $(n_l)$ such that $\|u_{n_l}-z_{n_l}\|>\epsilon_0$ for all $l\geq 1.$ Using earlier argument, there exist a subsequence $(n_{l_k})$ such that $u_{n_{l_k}}\to p$ and $z_{n_{l_k}}\to p$ for some $p\in \mathcal{P}$ and for all $k\geq 1.$ This is a contradiction. Conversely, assume $(d).$ Let $(\mathcal{P},\mathcal{Q})$ is a non-empty closed bounded convex pair in $\mathcal{E}.$ Let $\|u-v\|=d(\mathcal{P},\mathcal{Q})=\|z-v\|$ for some $u, z\in \mathcal{P}$ and $v\in \mathcal{Q}.$ Set $\mathcal{C}=$ $\overline{co}\{u,z\}.$ Then $(\mathcal{C},\{v\})$ is a compact convex pair in $\mathcal{E}$ with $d(\mathcal{P},\mathcal{Q})\leq d(\mathcal{C},\{v\})\leq \|u-v\|=d(\mathcal{P},\mathcal{Q}).$ By $(d),$ we have $z=u.$

\noindent $(b) \Leftrightarrow  (e):$ Assume $(b).$ Let $(\mathcal{P},\mathcal{Q})$ be a bounded convex semisharp proximal pair in $\mathcal{E}.$ Suppose $(\mathcal{P},\mathcal{Q})$ is weak approximatively compact, and $\mathcal{E}$ has the Kadec-Klee property. Let $(u_n),(z_n)$ in $\mathcal{P}$ and $(v_n)$ in $\mathcal{Q}$ be sequences such that $\|u_n- v_n\|\to d(\mathcal{P},\mathcal{Q})$ and $\|z_n-v_n\|\to d(\mathcal{P},\mathcal{Q}).$ It is clear that $(u_n), (z_n)$ in $\mathcal{P}$ and $(v_n)$ in $\mathcal{Q}$ are minimizing sequences.
Then there exists a subsequence $(n_k)$ of $(n)$ such that $u_{n_k}\to u, z_{n_k}\to z$ and $v_{n_k}\to v$ weakly for some $x, z\in \mathcal{P}$ and $v\in \mathcal{Q}.$ Then $u_{n_k}-v_{n_k}\to u-v$ and $z_{n_k}-v_{n_k}\to z-v$ weakly. Since the norm is weakly lower semicontinuous, we have $\|u_{n_k}-v_{n_k}\|\to \|u-v\|$ and $\|z_{n_k}-v_{n_k}\|\to \|z-v\|.$ By the Kadec-Klee property of $\mathcal{E},$ it is clear that $u_{n_k}-v_{n_k}\to u-v$ and $z_{n_k}-v_{n_k}\to z-v.$ Thus $u_{n_k}-z_{n_k}\to u-z.$ Moreover, $\|u-v\|=d(\mathcal{P},\mathcal{Q})=\|z-v\|.$ By $(b), u=z.$ Hence $\|u_{n_k}-z_{n_k}\|\to 0.$ By applying the earlier argument, we have $\|u_n-z_n\|\to 0.$
 To see $(e) \Rightarrow  (b),$ suppose $u,z\in \mathcal{P}$ and $v\in \mathcal{Q}$ such that $\|u-v\|=d(\mathcal{P},\mathcal{Q})=\|z-v\|.$ Set $\mathcal{P}=\overline{co}(\{u,z\})$ and $\mathcal{Q}=\{v\}.$ Then $(\mathcal{P},\mathcal{Q})$ is a closed bounded weak approximatively compact convex pair. By property UC, we have $u=z.$

\noindent $(a)\Leftrightarrow (f):$ Suppose $\mathcal{E}$ is strictly convex and $(\mathcal{P},\mathcal{Q})$ is a non-empty pair with $\mathcal{P}$ convex. For $u, z\in \mathcal{P}, u\neq z$ and $v\in \mathcal{Q}$ such that $\|u-v\|=d(v,\mathcal{P})=\|z-v\|.$ By strict convexity, $\left\|\frac{u+z}{2}-v\right\|<d(v,\mathcal{P}),$ a contradiction. Conversely, suppose that $\mathcal{E}$ is not strictly convex. Then there exist $ c_1, c_2\in S_X, c_1\neq c_2$ such that $\{tc_1+(1-t)c_2: 0\leq t\leq 1\}\subseteq S_\mathcal{E}.$ Set $\mathcal{C}=\{tc_1+(1-t)c_2: 0\leq t\leq 1\}$ and $\mathcal{D}=\{0\}.$ It is clear that $(\mathcal{C},\mathcal{D})$ contradicts $(f).$

\noindent $(a)\Leftrightarrow (g):$ Suppose $\mathcal{E}$ is not strictly convex. Then as mentioned above there exists a set $\mathcal{P}=\{tc_1+(1-t)c_2: 0\leq t\leq 1\}\subseteq S_\mathcal{E}$ where $c_1, c_2\in S_\mathcal{E}$ with $c_1\neq c_2.$ We see that all elements of $\mathcal{P}$ are of minimum norm. This contradicts $(g).$ This proves $(g)\Rightarrow (a).$ The proof for $(a)\Rightarrow (g)$ follows from Theorem \ref{Thm1:EldredRaj2014}. 

\noindent $(a)\Leftrightarrow (h):$ Suppose $\mathcal{E}$ is not strictly convex. Then as mentioned above there exists a set $\mathcal{Q}=\{tc_1+(1-t)c_2: 0\leq t\leq 1\}\subseteq S_\mathcal{E}$ where $c_1, c_2\in S_\mathcal{E}$ with $c_1\neq c_2.$  Make $\mathcal{P}=\{0\}.$ Then $\mathcal{P}_0=\mathcal{P}$ and $\mathcal{Q}_0=\mathcal{Q}.$ Moreover, $\|P_\mathcal{P}(c_1)-P_\mathcal{P}(c_2)\|=0\neq \|c_1-c_2\|.$ Thus $P_\mathcal{P}$ is not an isometry. This proves $(h)\Rightarrow (a).$
The proof for $(a)\Rightarrow (h)$ follows from Theorem \ref{Thm2:EldredRaj2014}.

\noindent $(a)\Leftrightarrow (i):$ Suppose $X$ is strictly convex and $(\mathcal{P},\mathcal{Q})$ is a non-empty convex pair of $\mathcal{E}.$ Let $v_1, v_2\in \mathcal{Q}$ and $u\in \mathcal{P}$ such that $v_1, v_2\in \mathbb{B}(u, d(\mathcal{P},\mathcal{Q})).$ Then $\|u-v_i\|= d(\mathcal{P},\mathcal{Q}), 1\leq i\leq 2.$ By strict convexity, $v_1=v_2.$ Conversely, suppose that $\mathcal{E}$ is not strictly convex. Then there is a set $\mathcal{C}=\{tc_1+(1-t)c_2: 0\leq t\leq 1\}\subseteq S_\mathcal{E}$ where $c_1, c_2\in S_\mathcal{E}$ with $c_1\neq c_2.$ Set $\mathcal{D}=\{0\}.$ Then $d(\mathcal{C},\mathcal{D})=1$ and the closed unit ball $\mathbb{B}(0,1)$ contains $\mathcal{C}$ in the boundary, contradicting $(i).$ 
\end{proof}

\section{On proximal uniform normal structure and best proximity pairs}

We fix a few notations for a non-empty bounded convex pair $(\mathcal{P},\mathcal{Q})$ of subsets of a Banach space $\mathscr{E}.$ 
For $u\in \mathcal{P},~r(u,\mathcal{Q}) =\sup\{\|u-v\|:v\in \mathcal{Q}\}.$ We denote $d(\mathcal{P},\mathcal{Q})=\inf\{\|z-w\|:z\in \mathcal{P}, w\in \mathcal{Q}\},~\delta(\mathcal{P},\mathcal{Q})=\sup\{\|u-v\|:u\in \mathcal{P}, v\in \mathcal{Q}\},~ r(\mathcal{P},\mathcal{Q})=\inf\{r(z, \mathcal{Q}):z\in \mathcal{P}\},$ $R(\mathcal{P},\mathcal{Q})=\max\{r(\mathcal{P},\mathcal{Q}), r(\mathcal{Q}, \mathcal{P})\},$ $\mathcal{P}_0=\{a\in \mathcal{P}: \|a-b\|=d(\mathcal{P},\mathcal{Q})~\mbox{for some}~b\in \mathcal{Q}\}$ and $\mathcal{Q}_0=\{b\in \mathcal{Q}: \|b-a\|=d(\mathcal{P},\mathcal{Q})~\mbox{for some}~a\in \mathcal{P}\}.$ Also, we denote $\overline{co}(\mathcal{P})$ by the closed and convex hull of $\mathcal{P}$.  
The pair $(\mathcal{P},\mathcal{Q})$ is proximal (resp., semisharp proximal) (\cite{Espinola2008, Raju2011}) if for $(u,v)\in (\mathcal{P},\mathcal{Q})$ there exists one (resp., at most one) pair $(z,w)\in (\mathcal{P},\mathcal{Q})$ for which $\|u-w\|=d(\mathcal{P},\mathcal{Q})=\|v-z\|$. If such a pair $(z,w)$ is unique for given any $(u,v) \in (\mathcal{P},\mathcal{Q})$, then $(\mathcal{P},\mathcal{Q})$ is said to be a sharp proximal pair and $(z,w)$ is denoted by $(v',u')$. We denote a collection $\Upsilon (\mathcal{P},\mathcal{Q})$ consisting of non-empty bounded closed convex proximal pairs $(A_1,B_1)$ of $(\mathcal{P},\mathcal{Q})$ with at least one of $A_1$ and $B_1$ contains more than one point and $d(A_1,B_1)=d(\mathcal{P},\mathcal{Q}).$ According to \cite{AbhikRaju2021}, the pair $(\mathcal{P},\mathcal{Q})$ is said to have proximal uniform normal structure (\cite[Definition 5.1, p. 9]{AbhikRaju2021}) if $N(\mathcal{P},\mathcal{Q})=\displaystyle \sup \frac{R(A_1,B_1)}{\delta(A_1,B_1)} < 1.$ 
Here the supremum is taken over all $(A_1,B_1)$ in $\Upsilon(\mathcal{P},\mathcal{Q}).$
 If $d(\mathcal{P},\mathcal{Q})=0$, then proximal uniform normal structure turns out to be uniform normal structure (which was introduced in \cite{GW-UNS-1979}) and the existence of the best proximity pair (\cite[Theorem 5.8, p. 11]{AbhikRaju2021}) boils down to the corresponding existing fixed point result (\cite[Theorem 7, p. 1234]{LImXu1995}). Hence, throughout this note, without loss generality we assume that $d(\mathcal{P},\mathcal{Q})>0$.

Let $(A_1,B_1)$ be an element in $\Upsilon(\mathcal{P},\mathcal{Q}).$ Then it is possible to construct a sequence $\{(A_n, B_n)\}$ in $\Upsilon(A_1,B_1)$ 
such that $\delta(A_n)=\delta(B_n)\to 0.$
 Consequently, $R(A_n, B_n)\leq R(A_n)+ d(\mathcal{P},\mathcal{Q})\leq \delta(A_n)+ d(\mathcal{P},\mathcal{Q})\to d(\mathcal{P},\mathcal{Q})$ and $\delta(A_n, B_n)\leq \delta(A_n)+ d(\mathcal{P},\mathcal{Q})\to d(\mathcal{P},\mathcal{Q}).$ Hence $N(\mathcal{P},\mathcal{Q})=1$ and so, no $(\mathcal{P},\mathcal{Q})$ can have proximal uniform normal structure. Thus to enrich the family of non-empty pairs that satisfy proximal uniform normal structure, we reform the notion.

 \begin{definition}\label{PNUS}
 A non-empty bounded convex pair $(\mathcal{P},\mathcal{Q})$ of a Banach space $\mathscr{E}$ is said to have proximal uniform normal structure if $$\mathcal{N}(\mathcal{P},\mathcal{Q})=\displaystyle\sup\left\{\frac{R(A_1, A_2)-d(A_1,A_2)}{\delta(A_1, A_2)-d(A_1,A_2)}:(A_1, A_2)\in\Upsilon (\mathcal{P},\mathcal{Q}) \right\}<1.$$
 \end{definition}
 
 It is easy to observe that if a non-empty bounded convex pair $(\mathcal{P},\mathcal{Q})$ has proximal uniform normal structure as defined in Definition \ref{PNUS}, then it has proximal normal structure (\cite{Eldred2005}). 
 We say that a Banach space $\mathscr{E}$ has proximal uniform normal structure if every non-empty closed bounded convex proximal pair of $\mathscr{E}$ has proximal uniform normal structure.
It is known (\cite{AbhikRaju2021}) that if a non-empty closed bounded convex pair in a Banach space has proximal uniform normal structure, then it is semisharp proximal. Thus 
in view of Theorem \ref{Thm2.1:StrictConvexityEquivalence}, we have the following theorem. 
 \begin{theorem}\label{Thm:StrictConvexity}
 Let $X$ be a Banach space having proximal uniform normal structure. Then $X$ is strictly convex.
 \end{theorem}
Let $(A,B)$ be as given in Example 5.2 of \cite{AbhikRaju2021}. It can be verified (similar computations as in Example 5.2 of \cite{AbhikRaju2021}) that for $(A_1,B_1)\in \Upsilon(A,B),$ there is a constant $c<1$ (independent of $(A_1,B_1)$) such that $\displaystyle\frac{R(A_1,B_1)^2-d(A_1,B_1)^2}{\delta(A_1,B_1)^2-d(A_1,B_1)^2}= \frac{R(B_1,B_1)^2}{\delta(B_1,B_1)^2}\leq \displaystyle \frac{c^2 \delta(B_1,B_1)^2}{\delta(B_1,B_1)^2}.$ Hence $(A,B)$ has proximal uniform normal structure.

It should be observed that Lemma 5.5 of \cite{AbhikRaju2021} holds true with the revised notion of proximal uniform normal structure. We mention that we need to modify Lemma 5.6 and Proposition 5.7 of \cite{AbhikRaju2021}. 
 \begin{proposition}\label{propertyP}
 Let $(\mathcal{P},\mathcal{Q})$ be a non-empty weakly compact and convex pair in a Banach space $\mathscr{E}.$ Suppose $\{z_n\}$ and $\{u_n\}$ are two sequences in $\mathcal{P}$ and $\{v_n\}$ is a sequence in $\mathcal{Q}.$ Then there exists $v_0\in \mathcal{Q}$ such that $\displaystyle\max\{\limsup_{n}\|z_n-v_0\|, \limsup_{n}\|u_n-v_0\|\} \leq \limsup_{i}\max\{\limsup_{n}\|z_n-v_i\|, \limsup_{n}\|u_n-v_i\|\}.$
 \end{proposition}
 \begin{proof}
 Set $C = \displaystyle\bigcap_{n=1}^{\infty}\overline{co}\left(\{v_i:i\geq n\}\right).$
 By weak compactness $C\neq \emptyset$. As the function $h_1(x)=\displaystyle\limsup_{n}\|z_n-x\|$ and $h_2(x)=\displaystyle\limsup_{n}\|u_n-x\|$ are weakly lower semicontinuous, the function $h(x)=\max\{h_1(x), h_2(x)\}$ is pointwise weakly lower semicontinuous. Thus there exist a subsequence $\{v_{k_i}\}$ and a point $v_0$ in $C$ such that $$\displaystyle h(v_0)\leq \liminf_{i} h(v_{k_i})\leq \limsup_{i} h(v_i)=\limsup_{i}\max\{h_1(v_i), h_2(v_i)\}.$$ This completes the proof.
 \end{proof}
 \begin{lemma}\label{Main Lemma 1}
Let $(\mathcal{P},\mathcal{Q})$ be a non-empty pair as in Proposition \ref{propertyP}.
 Suppose $(\mathcal{P},\mathcal{Q})$ has proximal uniform normal structure. 
  Let $\{u_n\}$ be a sequence in $\mathscr{E}$ such that $\{u_{2n}\}\subseteq \mathcal{P}$ and $\{u_{2n+1}\}\subseteq \mathcal{Q}$. Then there exists $x\in \mathcal{P}_0$ such that
  \begin{itemize}
 \item[(i)] $\displaystyle\max\{\limsup_{n} \left\|x-u_{2n}'\right\|, \limsup_{n} \left\|x-u_{2n+1}\right\|\}-d(\mathcal{P},\mathcal{Q})\\
 \leq \mathcal{N}(\mathcal{P},\mathcal{Q})\left(\delta\left(\{u_{2n}, u_{2n+1}'\}, \{u_{2n}', u_{2n+1}\}\right)-d(\mathcal{P},\mathcal{Q})\right);$
 \item[(ii)] $\|x-v\|\leq \displaystyle \max\{\limsup_{j}\|v-u_{2j}\|, \limsup_{j}\|v-u_{2j+1}'\|\}$ for all $v\in \mathcal{Q};$
 \end{itemize}
 \end{lemma}
 \begin{proof}
  Set $\mathcal{P}_{2l}=\overline{co}\left(\{u_{2j}, u_{2j+1}' :j\geq l\}\right)$ and $\mathcal{P}_{2l}'=\overline{co}\left(\{u_{2j+1}, u_{2j}':j\geq l\}\right),$ $l\geq 1$. Then $\{\mathcal{P}_{2l}\}$ (similarly, $\{\mathcal{P}_{2l}'\}$) is a decreasing sequence of non-empty closed and convex subsets of $\mathcal{P}.$ By weak compactness, the set $\mathcal{P}^{0}=\displaystyle \bigcap_{l=1}^{\infty}\mathcal{P}_{2l}\neq \emptyset.$ 
 For any $x\in \mathcal{P}^{0}$ and $v\in \mathcal{Q},$ $\|v-x\|\leq r(v, \mathcal{P}_{2l})=\displaystyle \max\{\sup_{j\geq l}\|v-u_{2j}\|, \sup_{j\geq l}\|v-u_{2j+1}'\|\}.$ Hence $\|v-x\|\leq \displaystyle \max\{\limsup_{j}\|v-u_{2j}\|, \limsup_{j}\|v-u_{2j+1}'\|\}.$ This proves (ii).
 To prove (i), without loss of any generality, let's assume that $\delta(\mathcal{P},\mathcal{Q})>d(\mathcal{P},\mathcal{Q}).$ For $l\in \mathbb{N},$ choose $\epsilon _l>0$ such that $\displaystyle \lim_{l\to \infty} \epsilon_l =0.$ As $\delta\left(\mathcal{P}_{2l},\mathcal{P}_{2l}'\right)\leq \delta\left(\{u_{2l}, u_{2l+1}'\}, \{u_{2l}', u_{2l+1}\}\right)$, for $l\geq 1,$ there exists $x_{2l}\in \mathcal{P}_{2l}$ such that
$ r(x_{2l},\mathcal{P}_{2l}') -d(\mathcal{P},\mathcal{Q})
 < R(\mathcal{P}_{2l},\mathcal{P}_{2l}')-d(\mathcal{P},\mathcal{Q})+\epsilon_l.$
 Hence 
\begin{eqnarray*}
 \max\{\displaystyle\limsup_{j}\left\|x_{2l}-u_{2j}'\right\|, \limsup_{j}\left\|x_{2l}-u_{2j+1}\right\|\}-d(\mathcal{P},\mathcal{Q})\\
  \leq \mathcal{N}(\mathcal{P},\mathcal{Q}) \left(\delta \left(\{u_{2l}, u_{2l+1}'\}, \{u_{2l}', u_{2l+1}\}\right)-d(\mathcal{P},\mathcal{Q})\right) +\epsilon_l.
\end{eqnarray*} 
 As $F=\displaystyle \bigcap_{j=1}^{\infty}\overline{co}\{x_{2i}:i\geq j\}\neq \emptyset,$ by Proposition \ref{propertyP}, there is $w\in F$ satisfying 
\begin{eqnarray*}
&&\displaystyle \max\{\limsup_{j}\|w-u_{2j}'\|, \limsup_{j}\|w-u_{2j+1}\|\}-d(\mathcal{P},\mathcal{Q})\\
&&\leq \mathcal{N}(\mathcal{P},\mathcal{Q}) (\delta \left(\{u_{2l}, u_{2l+1}'\}, \{u_{2l}', u_{2l+1}\}\right)-d(\mathcal{P},\mathcal{Q})).
\end{eqnarray*} 
  This proves (i). 
 \end{proof} 
 Let $(\mathcal{P}, \mathcal{Q})$ be a non-empty closed bounded convex proximal pair in a Banach space $\mathcal{E}.$ Suppose $\mathcal{E}$ has proximal uniform normal structure. Let $x,y\in \mathcal{P}.$ By strict convexity (using Theorem \ref{Thm:StrictConvexity}) we have $\|x-y'\|=\|x'-y\|.$ In the below we use this fact to prove the main theorem of \cite{AbhikRaju2021}.

\noindent {\it Revised Proof of Theorem 5.8 of \cite{AbhikRaju2021}.} 
% \begin{proof}
Replace $(A,B)$ by $(\mathcal{P}, \mathcal{Q})$ in Theorem 5.8 of \cite{AbhikRaju2021}.
 Let $u_0\in \mathcal{P}_0.$ Then $\{T^{2m}u_0\}\subseteq \mathcal{P}_0,$ $ \{T^{2m+1}u_0\}\subseteq \mathcal{Q}_0.$ Then there exists $u_1\in \mathcal{P}_0$ that satisfies Lemma \ref{Main Lemma 1} by replacing $x$ by $u_1$ and $(u_{2m}, u_{2m+1})$ by $\left(T^{2m}u_0, T^{2m+1}u_0\right).$ By induction, we can find a sequence $\{u_i\}$ in $\mathcal{P}_0$ such that, for $i\in \mathbb{N},$
 {\footnotesize{\begin{align}\label{eqn4.3.4}
\begin{split}
 & \displaystyle\max\{\displaystyle \limsup_{m} \|u_i- T^{2m}u_{i-1}'\|,\limsup_{m} \|u_i- T^{2m+1}u_{i-1}\|\}-d(\mathcal{P},\mathcal{Q})\\ & \leq \mathcal{N}(\mathcal{P},\mathcal{Q})\left(\delta\big(\{T^{2m}u_{i-1}, T^{2m+1}u_{i-1}'\}, \{T^{2m}u_{i-1}', T^{2m+1}u_{i-1}\}\big)-d(\mathcal{P},\mathcal{Q})\right);\\ 
& \|u_i-v\|\leq \max\{\displaystyle\limsup_{m}\|T^{2m}u_{i-1}-v\|, \limsup_{m}\|T^{2m+1}u_{i-1}'-v\|\}~~\mbox{for all}~v~\mbox{in}~\mathcal{Q}.
\end{split}
\end{align}}}
 Set $g_i(\mathcal{P})=\max\{\displaystyle \limsup_{j} \|u_{i+1}- T^{2j}u_{i}'\|,\limsup_{j} \|u_{i+1}- T^{2j+1}u_{i}\|\}.$ 
 Using (\ref{eqn4.3.4}) and similar techniques used in the proof of Theorem 5.8 of \cite{AbhikRaju2021}, we get $$g_i(\mathcal{P})-d(\mathcal{P},\mathcal{Q})\leq k^{2}\left[g_{i-1}(\mathcal{P})-d(\mathcal{P},\mathcal{Q})\right]\leq \cdots \leq \left(\mathcal{N}(\mathcal{P},\mathcal{Q})\cdot k^2\right)^i [g_{0}(\mathcal{P})-d(\mathcal{P},\mathcal{Q})].$$
Thus, $\max\left\{\displaystyle\lim_{i\to \infty} g_i(\mathcal{P}), \displaystyle \lim_{i\to \infty}\delta\big(\{T^{2m}u_{i}, T^{2m+1}u_{i}'\}, \{T^{2m}u_{i}', T^{2m+1}u_{i}\}\big)\right\}$ $ =d(\mathcal{P},\mathcal{Q}).$ Substituting $v=u_{i}'$ in (\ref{eqn4.3.4}), we have
 $\displaystyle\limsup_{m}\|T^{2m}u_{i-1}-u_{i}'\|\to d(\mathcal{P},\mathcal{Q})$ and $\displaystyle \limsup_{m}\|T^{2m+1}u_{i-1}'-u_{i}'\|\}\to d(\mathcal{P},\mathcal{Q})$ as $i\to \infty.$
  By weak compactness, $ K =\displaystyle \bigcap_{j=1}^{\infty}\overline{co}(\{u_i':i\geq j\})\neq \emptyset.$
  By Proposition \ref{propertyP}, there exists $w\in K$ and a subsequence $\{u_{i_j}\}$ such that $\displaystyle \lim_{j\to \infty} \|w-u_{i_j}\| = d(\mathcal{P},\mathcal{Q}).$ By a similar technique used in the proof of Theorem 5.8 of \cite{AbhikRaju2021}, we get a point $z_0\in H=\displaystyle \bigcap_{j=1}^{\infty}\overline{co}(\{u_{i_j}:i\geq j\})$ such that $\|z_0-Tz_0\|=d(\mathcal{P},\mathcal{Q}).$

%\section{Bibliography styles}
%
%There are various bibliography styles available. You can select the style of your choice in the preamble of this document. These styles are Elsevier styles based on standard styles like Harvard and Vancouver. Please use Bib\TeX\ to generate your bibliography and include DOIs whenever available.
%
%Here are two sample references: \cite{Feynman1963118,Dirac1953888}.
%
%\section*{References}

%\bibliography{mybibfile}

\end{document}